\documentclass[reqno,11pt]{amsart} 
    \usepackage{amsmath,amscd,amsfonts,amssymb}
    \usepackage{mathrsfs,dsfont}
    \usepackage{color}
    \usepackage{mathtools}
    \usepackage{hyperref}
    \usepackage{tikz-cd}
    \usepackage[normalem]{ulem}
    \usepackage{xcolor}
    \usepackage{enumitem}

    \usepackage{tikz}
    \usetikzlibrary{decorations.pathreplacing}
    
    \numberwithin{equation}{section}
    \numberwithin{figure}{section}
    
    \def\R{\mathbb{R}}

    \def\eps{\varepsilon}

    \renewcommand\leq{\leqslant}
    \renewcommand\geq{\geqslant}

    \theoremstyle{plain}
    \newtheorem{thm}{Theorem}[section]
    \newtheorem{theorem}[thm]{Theorem}
    
    \newtheorem{lemma}[thm]{Lemma}

    \newtheorem{proposition}[thm]{Proposition}

    \newtheorem*{claim*}{Claim}

    \theoremstyle{definition}
    \newtheorem{definition}[thm]{Definition}
    \newtheorem*{definition*}{Definition}
    \newtheorem*{remarks*}{Remarks}
    \newtheorem*{remark*}{Remark}
    \newtheorem{remark}[thm]{Remark}

    \newcommand{\rank}{\mathrm{rank}}

    \newcommand{\Kak}{\mathrm{Kak}}
    \newcommand{\lin}{\mathrm{lin}}
    
    \renewcommand{\H}{\mathrm{H}}

\begin{document}

	\title{Lifting curved Kakeya sets to linear Kakeya sets}
 
	\author{Arian Nadjimzadah}
	\address{UCLA Department of Mathematics, Los Angeles, CA 90095.}
	\email{anad@math.ucla.edu}

	\date{}
	

    \begin{abstract}
        We prove that many curved Kakeya problems, originally motivated by H\"ormander's oscillatory integral problem, lift to the classical Kakeya problem in higher dimensions.

        As a consequence, if the classical Kakeya set conjecture were true in all dimensions, 
        the families of curves whose curved Kakeya sets have full dimension are dense among H\"ormander-type families. This is in contrast to the nowhere dense Bourgain's condition, which is a necessary condition for best-case curved Kakeya maximal function estimates. 
        Furthermore the linear Kakeya set conjecture would imply a fairly complete understanding of Kakeya sets of quadratic curves, as first studied systematically by Wisewell. These results give a better understanding of a question posed by Guo--Guth--Nadjimzadah--Shen--Zhang, and they more generally show that the classical Kakeya conjecture in high dimensions rests on a broad range of curved Kakeya problems in lower dimensions. 
    \end{abstract}
	
	\maketitle

    \section{Introduction}

    The Kakeya set conjecture in $\R^n$ asserts that a compact set with a line segment in each direction has Hausdorff dimension $n$. 
    Wang and Zahl recently proved this conjecture in $\R^3$ \cite{wang2025volumeestimatesunionsconvex}. The geometry of lines in $\R^3$ is just special enough that their proof did not need to grapple with any curved objects. 
    However already in $\R^4$, there are certain quadric surfaces that are ``near misses'' to the Kakeya conjecture \cite{guth2025introductionproofkakeyaconjecture, ZahlSurvey}, so one seems essentially forced to establish Kakeya-type estimates for curved objects. 
    
    The purpose of this paper is to make this intuition concrete,
    and show that strong results for Kakeya sets of curves would follow from the resolution of the classical \emph{linear} Kakeya set conjecture in higher dimensions. These Kakeya sets of curves are motivated by H\"ormander's oscillatory integral problem, and they may have exotic behavior as was first discovered by Bourgain. We will see that under the assumption of the linear Kakeya set conjecture in all dimensions, the families of curves whose curved Kakeya sets have full dimension are dense in the H\"ormander-type families. 
    At the same time Bourgain's condition, which is closed nowhere dense, is necessary for optimal curved Kakeya maximal function estimates. Conditional on linear Kakeya, we also fully classify the minimal dimensions of Kakeya sets of quadratic curves in $\R^n$, as first studied in depth by Wisewell \cite{wisewell}, in particular showing that they always take on integer values. 
    Both these results support a question posed in \cite{guo2026curvedkakeyaproblemsprojective}.

    Our strategy is to lift a Kakeya set of curves to a Kakeya set of lines by lifting each of its curves to a ruled surface. Roughly speaking, the more ``complicated'' a family of curves is, the more dimensions one requires to lift its Kakeya sets to linear Kakeya sets. 
    Before giving the formal definitions and statements of results, we present a simple example that illustrates the mechanism.

    \subsection{A basic example of the lifting mechanism}\label{subsec: worked out example}

    Consider the $4$-parameter family of curves
    \begin{align}
        \gamma_{y,w} = \{(w_1 + t y_1+ t^2y_1, w_2 + t y_2 ,t) : |t| \leq 1/10\} \subset \R^3,
    \end{align}
    where the ``direction'' is $y$, and $|y| \leq 1/10, |w| \leq 10$. Without the perturbation ``$t^2y_1$'', this would be the standard family of lines in the linear Kakeya problem. 
    A corresponding curved Kakeya set $K$ contains a curve $\gamma_{y,w(y)}$ for each $y$. 
    Let us now see how to lift the curved Kakeya set $K \subset \R^3$ to a linear Kakeya set $\widetilde K \subset \R^4$. 

    Define the submersion $\pi : \R^4 \to \R^3$ by $\pi(x_1,x_2,s,t)=(x_1+ts,x_2,t)$ and the 5-parameter family of line segments
    \begin{align}
        L_{y,w,r} = 
        \{(w_1,w_2,r,0) +t(y_1-r,y_2,y_1,1) : |t| \leq 1/10\} \subset \R^4.
    \end{align}
    The \emph{ruling} parameter is $r \in [-1/10,1/10]$.
    The sets $\pi^{-1}(\gamma_{y,w})$ are 2-dimensional quadric surfaces ruled by
    $\bigcup_{r} L_{y,w,r}$. For a compact set $\Omega_0 \subset \R^4$ large enough to contain the family of line segments, define the lifted set
    \begin{align}
        \widetilde K = \pi^{-1}(K) \cap \Omega_0 \subset \R^4.
    \end{align}
    Since $K$ is a Kakeya set of the curves $\gamma_{y,w}$,
    \begin{align}
        \widetilde K \supset \bigcup_y (\pi^{-1} (\gamma_{y,w(y)})\cap \Omega_0 )
        \supset \bigcup_{y,r} L_{y,w(y),r}.
    \end{align}
    The slope map $(y_1,y_2,r) \mapsto (y_1-r,y_2,y_1)$ has full rank. Hence the lifted lines have an open set of directions in $\R^4$, and finitely many rotated copies of $\widetilde K$ form a linear Kakeya set. 
    Since $\pi$ is a submersion with $1$-dimensional fibers,  
    \begin{align}
        \dim_\H \widetilde K \leq \dim_\H K + 1.
    \end{align}
    Thus if the linear Kakeya set conjecture held in $\R^4$, these curved Kakeya sets would have full Hausdorff dimension $3$. 
    
    Figure \ref{fig: lift cartoon} is a cartoon of this mechanism, drawn in one lower dimension.
    \begin{figure}
        \centering
        \includegraphics[width=0.5\linewidth]{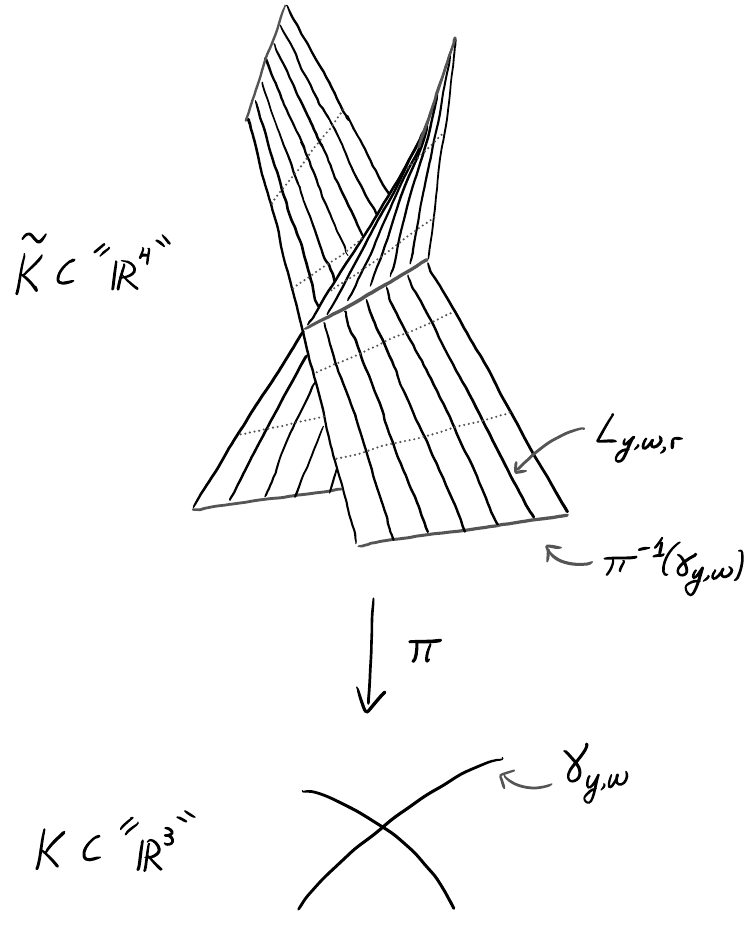}
        \caption{Cartoon of the lifting mechanism.}
        \label{fig: lift cartoon}
    \end{figure}

    \subsection{Main results}

    We use the definition of H\"ormander-type families of curves from \cite{guo2026curvedkakeyaproblemsprojective}, which is a natural generalization of the families of characteristic curves that arise in H\"ormander's oscillatory integral problem. Let $n \geq 2$ and fix compact balls $W,Y \subset \R^{n-1}$ and a compact interval $I \subset \R$. Set $Q = W \times I \times Y$ and consider the space $C^\infty(Q ; \R^{n-1})$ equipped with the Fr\'echet topology induced by the norms
    \begin{align}
        \|X \|_{C^r(Q)} := \max_{|\alpha| \leq r} \| \partial^\alpha  X\|_{L^\infty(Q)}.
    \end{align}
    Let $\mathbf H \subset C^\infty(Q; \R^{n-1})$ be the set of maps satisfying \emph{H\"ormander's conditions}:
    \begin{align}
        \det\nabla_w X(w,t;y) \neq 0, \quad \det \partial_t [(\nabla_w X)^{-1} \partial_y X](w,t;y) \neq 0
    \end{align}
    on $Q$. This is a slight reformulation of \cite[Definition 1.14]{guo2026curvedkakeyaproblemsprojective}, and one can see that they are equivalent by using the Schur complement.
    The set $\mathbf H$ is open in $C^\infty(Q; \R^{n-1})$. 
    A map $X \in \mathbf H$ defines a $2(n-1)$-parameter family of curves $\mathcal C(X)$ consisting of 
    \begin{align}
        \gamma_{y,w} = \{(X(w,t;y),t) : t \in I\} \subset \R^{n}
    \end{align}
    for $(y,w) \in Y \times W$. We write $\gamma_{y,w}(t):=(X(w,t;y),t)$ for their parameterizations. 
    We call $y$ the direction of the curve. We will often add the superscripts $(n)$ to $\mathbf H$, $X$, and related objects, to remove any ambiguity as to the dimension $n$. 
    
    We say that a compact set $K \subset \R^n$ is an $X$-\emph{Kakeya set} if for each $y \in Y$, there is $w(y) \in W$ such that $\gamma_{y,w(y)} \subset K$.
    Define the minimal Hausdorff dimension 
    \begin{align}
        d_{\Kak}(X) = \inf\{\dim_\H K : \text{$K$ is an $X$-Kakeya set}\}.
    \end{align}
    By the curved version of Bourgain's bush argument in \cite{wisewell} and the even-dimensional curved Kakeya result in \cite{bourgainGuth}, 
    \begin{align}\label{eq: universal bound}
        d_{\Kak}(X) \geq 
        \begin{cases}
            (n+1)/2 & n \text{ odd,}\\
            (n+2)/2 & n \text{ even}.
        \end{cases}
    \end{align}
    Equality is attained by explicit quadratic families; see Remark \ref{rmk: wisewell}; see also \cite{guthOscillatory} for a construction with cubic curves.
    There is a related Kakeya maximal function,
    \begin{align}
        \mathcal K_{\delta}^{X} f(y) = \sup_{w \in W}\frac{1}{|N_\delta(\gamma_{y,w})|}\int_{N_\delta(\gamma_{y,w})} |f|,  \quad y \in Y,
    \end{align}
    where $N_\delta(\gamma_{y,w})$ denotes the $\delta$-neighborhood of $\gamma_{y,w}$.
    By understanding the $L^p$ bounds of these operators, one can make progress on oscillatory integral problems, and the framework includes several further maximal functions of interest in their own right; see \cite{guo2026curvedkakeyaproblemsprojective} for further details.

    The linear Kakeya problem arises from the choice $X_{\lin}(w,t;y) = w + ty$, and the linear Kakeya set conjecture in $\R^n$ asserts that 
    $d_\Kak(X_{\lin}^{(n)}) = n$. The stronger Kakeya maximal function conjecture asserts 
    \begin{align}\label{eq: lin Kak max est}
        \|\mathcal K_\delta^{X_{\lin}} f\|_{L^n(Y)} \lesssim_\eps \delta^{-\eps} \| f\|_{L^n(\R^n)}.
    \end{align}
    Here and below, all parameter domains are fixed compact sets, and the implicit constants may depend on them. We write $A \lesssim_\eps B$ when $A \leq C(\eps) B$. 
    Since linear Kakeya sets play a distinguished role in this paper, we define 
    \begin{align}
        d_{\lin}(n) := d_{\Kak}(X_{\lin}^{(n)}), \qquad \mathcal K^{\lin, n}_\delta := \mathcal K_{\delta}^{X_{\lin}^{(n)}}.
    \end{align}
    For most $X$, the critical $L^n \to L^n$ maximal function estimate of the form \eqref{eq: lin Kak max est} fails, as Bourgain first showed and as later work generalized \cite{BourgainSeveralVariables, hormanderDichotomy, guo2026curvedkakeyaproblemsprojective}.
    In fact, \emph{Bourgain's condition} is necessary for this critical estimate. 

    \begin{definition}[Bourgain's condition \protect{\cite{guo2026curvedkakeyaproblemsprojective}}]
        We say that $X \in \mathbf H$ satisfies Bourgain's condition if there exists a scalar function $\lambda(w,t,y)$ such that 
        \begin{align}
            \partial_t^2((\nabla_w X)^{-1} \nabla_y X) = \lambda \partial_t ((\nabla_w X)^{-1} \nabla_y X) 
        \end{align}
        on $Q$.
        Let $\mathbf B \subset \mathbf H$ be the set of maps satisfying Bourgain's condition. 
    \end{definition}
    It is conjectured that Bourgain's condition is also sufficient for optimal maximal function estimates \cite{Nadjimzadah2026Bourgain, hormanderDichotomy}, and in particular for $d_{\Kak}(X^{(n)})=n$.

    Our first result says that, conditional on the linear Kakeya set conjecture in all dimensions, Bourgain's condition is far from necessary for the full-dimensional curved Kakeya set conclusion.

    \begin{theorem}\label{thm: density}
            Suppose that the linear Kakeya set conjecture holds in all dimensions. Fix $n \geq 3$. Then there is a dense set $\mathbf P^{(n)} \subset \mathbf H^{(n)}$ such that, for each $X \in \mathbf P$, every $X$-Kakeya set has Hausdorff dimension $n$. 

            Moreover, $\mathbf B \subset 
            \mathbf H$ is closed nowhere dense. Consequently $\mathbf P \setminus \mathbf B$ is dense in $\mathbf H$, and for each $X \in \mathbf P \setminus \mathbf B$, every $X$-Kakeya set has Hausdorff dimension $n$, whereas the critical estimate 
            \begin{align}
                \|\mathcal K_\delta^{X} f\|_{L^n(Y)} \lesssim_\eps \delta^{-\eps} \| f\|_{L^n(\R^n)}
            \end{align}
            fails.
    \end{theorem}
    For $X \in \mathbf H \setminus \mathbf B$, the critical Kakeya maximal estimate fails because there are $\delta$-discretized $X$-Kakeya sets with small portions that can compress near a surface. However one often has access to at least a weak form of the polynomial Wolff axioms that prevents global compression near a surface; see for example \cite{guo2026curvedkakeyaproblemsprojective, beyondUniversal}.
    Theorem \ref{thm: density} says that, if one believes linear Kakeya in higher dimensions, then at least in $\mathbf P \setminus \mathbf B$ there is no way to coordinate this \emph{local} compression
    in many places in some ``fractal way'' to obtain an $X$-Kakeya set of Hausdorff dimension $< n$. This supports the intuition that a counterexample to a curved Kakeya problem should arise only for ``algebraic'' reasons.

    Theorem \ref{thm: density} supports a question posed in \cite[Section 1.9]{guo2026curvedkakeyaproblemsprojective} as well. It seems plausible that $d_{\Kak}(X^{(n)}) = n$ is a \emph{generic} property (in the sense that it holds on a dense open set) conditional on the linear Kakeya set conjecture in all dimensions. 
    For odd $n$, Guo--Liu--Xi already gave an improvement over the universal bound \eqref{eq: universal bound} for generic families of curves coming from H\"ormander-type phase functions \cite{guo2025curvedkakeyasetsgeneric};
    it seems plausible that their argument can be adapted to the larger class $\mathbf H$.

    When we restrict to quadratic families of curves, we can say much more about the corresponding Kakeya sets. Let $n \geq 3$ and $B$ an $(n-1) \times (n-1)$ matrix. Consider the defining function $X_B : W_B \times I_B \times [-1,1]^{n-1} \to \R^{n-1}$ given by 
    \begin{align}\label{eq: quadratic def}
        X_B(w,t;y) = (w + (tI + t^2 B)y).
    \end{align}
    Here $W_B$ is a compact ball chosen large enough to make the construction of the compression examples in Theorem \ref{thm: quadratic} clean, and $I_B = [-\tau_B, \tau_B]$ is chosen small enough so that $X_B \in \mathbf H$. 
    Quadratic curves are general enough to exhibit Bourgain's compression phenomenon, and the study of Kakeya sets of quadratic curves was initiated by Wisewell \cite{wisewell}.

    Our second result shows that, conditional on linear Kakeya in all dimensions, we can compute the minimal dimensions $d_{\Kak}(X_B^{(n)})$ exactly.

    \begin{theorem}\label{thm: quadratic}
        Let $B$ be a $(n-1) \times (n-1)$ matrix. If the linear Kakeya set conjecture holds in dimension $n + \rank (B^2)$, then 
        \begin{align}
            d_{\Kak}(X_B^{(n)}) = n-\rank B + \rank(B^2).
        \end{align}
        The minimum is attained by an $X_B$-Kakeya set contained in a smooth variety of dimension $n-\rank B + \rank(B^2)$.
        In particular, conditional on linear Kakeya set conjecture in all dimensions, the numbers $d_{\Kak}(X^{(n)}_B)$ are integers. 
    \end{theorem}

    \begin{remark}\label{rmk: wisewell}
        In the special case $B^2 = 0$, Wisewell already found $X_B$-Kakeya sets that show $d_{\Kak}(X_B^{(n)}) \leq n - \rank B$ \cite[Proposition 29]{wisewell}.
    \end{remark}
    
    \begin{remark}
        Sometimes one can use a lower-dimensional linear Kakeya input than in Theorem \ref{thm: quadratic}. For example if $B$ is a nonzero multiple of $I$, one can check that $\mathcal C(X_B)$ is equivalent to the family of lines $\mathcal C(X_{\lin})$ after a diffeomorphism in $t$, so one can apply linear Kakeya in dimension $n$ instead of in dimension $n + \rank(B^2) = 2n-1$. 
    \end{remark}

    \begin{remark}
        Theorem \ref{thm: quadratic} recovers the universal lower bounds in \eqref{eq: universal bound}. If $n = 2k + 1$, take $B$ to be the direct sum of $k$ copies of the $2\times 2$ nilpotent Jordan block $J_2(0)$. Then $\rank B = k$ and $B^2 = 0$, so $d_{\Kak}(X_B) = k + 1 = (n+1)/2$. If $n = 2k$, take $B$ to be the direct sum of $k-1$ copies of $J_2(0)$ and one $1 \times 1$ zero block; then $d_\Kak(X_B) = k + 1 = (n+2)/2$.
    \end{remark}

    Our final result concerns the $\tan$-example from \cite{Nadjimzadah2026Bourgain}. Define the map 
    \begin{align}
        &X^{(n)}_{\tan} : [-0.1,0.1]^{n-1} \times [0.9,1.1] \times [-0.1,0.1]^{n-1} \to \R^{n-1}, \\
        &X_{\tan}(w,t;y) = (w'-t^2 y', \tan^{-1}(w_{n-1}/t)-ty_{n-1}), \label{eq: tan example def}
    \end{align}
    where $w=(w',w_{n-1})$ and $y=(y',y_{n-1})$.
    The map $X_{\tan}$ comes from a phase function $\phi_{\tan}$ satisfying Bourgain's condition, so $X_{\tan} \in \mathbf B$. We are able to show that $X^{(n)}_{\tan}$ lifts to a linear Kakeya set in $\R^{2n}$. 

    \begin{theorem}\label{thm: tan example main}
        Fix $n \geq 3$.
        If the linear Kakeya set conjecture holds in dimension $2n$, then $d_{\Kak}(X^{(n)}_{\tan}) = n$.
    \end{theorem}

    Note that one cannot directly apply linear Kakeya in $\R^n$, because $\mathcal C(X_{\tan})$ is not diffeomorphic to a family of lines \cite{Nadjimzadah2026Bourgain}. 
    Theorem \ref{thm: tan example main} also complements the main result of \cite{Nadjimzadah2026Bourgain}: when an $X_{\tan}$-Kakeya set is additionally \emph{sticky}, one can take advantage of a linear Kakeya estimate in dimension $n$---one does not need to reach for dimension $2n$. As we will see, the lift to dimension $2n$ takes advantage of the special structure of $X_{\tan}$, so it is not clear whether one should expect a similar theorem to hold for each $X \in \mathbf B$. 

    Since $X_{\tan} \in \mathbf B$, it is conjectured that the optimal $L^n \to L^n$ maximal function estimate holds. This raises an interesting question: can one use the linear Kakeya maximal estimate in higher dimensions to prove the $L^n \to L^n$ maximal function estimate for $X_{\tan}$? More ambitiously, could the restriction conjecture in higher dimensions imply the optimal $L^p$ bounds for the oscillatory integral operator with the phase function $\phi_{\tan}$? We believe this may be an interesting direction.

    \subsection{Outline}

    In Section \ref{sec: general lifting} we set up the lifting mechanism in general, state its dimension and maximal function consequences (Proposition \ref{prop: lift consequences}), and state the examples of lifts that lead to Theorems \ref{thm: density}, \ref{thm: quadratic}, and \ref{thm: tan example main} (Proposition \ref{prop: lifting examples}). Theorem \ref{thm: tan example main} is proved there too. 
    In Section \ref{sec: prop: lift consequences} we prove Proposition \ref{prop: lift consequences}. In Section \ref{sec: prop: lifting examples} we prove Proposition \ref{prop: lifting examples}. In Section \ref{sec: thm: density} we prove Theorem \ref{thm: density}. In Section \ref{sec: thm: quadratic} we prove Theorem \ref{thm: quadratic}. 

    \noindent \textbf{Acknowledgements.} 
    The author would like to thank Shaoming Guo, Terence Tao, Hong Wang, and Ruixiang Zhang for their encouragement.
    The author was supported in part by NSF DMS-2347850.

    \noindent \textbf{AI use statement.}
    The central idea of lifting curved Kakeya sets to higher-dimensional linear Kakeya sets is due to the author. The author subsequently used OpenAI's ChatGPT (GPT-5.6 Pro) to 
    explore examples of the lifting construction. 
    ChatGPT generated initial constructions for the examples in
    Proposition \ref{prop: lifting examples}, which the author
    simplified and refined into the form presented in
    Section \ref{sec: prop: lifting examples}.
    ChatGPT was also used for proofreading and revision.
    The author takes full responsibility for the content and correctness of the paper.

    \section{The general lifting mechanism}\label{sec: general lifting}

    We now describe the lifting mechanism which generalizes the example in Section \ref{subsec: worked out example}, and leads to Theorems \ref{thm: density}, \ref{thm: quadratic}, and \ref{thm: tan example main}.

    \begin{definition}[$(s,\nu)$-lift of $X^{(n)}$]
        We say that $X^{(n)}$ has an $(s,\nu)$-lift if the following holds. Define $N = n+s$. There exist
        \begin{enumerate}[label=(\roman*)]
            \item a compact ball $R \subset \R^s$,
            \item an open set $\Omega \subset \R^{N-1} \times \R$ and a smooth submersion 
            \begin{align}
                \pi : \Omega \to \R^{n-1} \times \R, \quad \pi(z,t) = (\pi_t(z), t),
            \end{align}
            \item smooth maps
            \begin{align}
                a : Y \times W \times R \to \R^{N-1}, \quad \Theta : Y \times R \to \R^{N-1},
            \end{align}
            such that, for every $(y,w,\rho,t) \in Y \times W \times R \times I$,
            \begin{align}
                L_{y,w,\rho}(t):=(a(y,w,\rho) + t\Theta(y,\rho), t) \in \Omega
            \end{align}
            and 
            \begin{align}\label{eq: projection property}
                \pi(L_{y,w,\rho}(t)) = \gamma_{y,w}(t).
            \end{align}
            Define $L_{y,w,\rho} = \{L_{y,w,\rho}(t) : t \in I\}$.
            Additionally on $Y \times R$,  
            \begin{align}\label{eq: direction rank}
                \rank D\Theta(y,\rho)= N-1-\nu.
            \end{align}
        \end{enumerate}
    \end{definition}
    We call $s$ the codimension of the lift and $\nu \geq 0$ the direction loss. The example in Section \ref{subsec: worked out example} is a $(1,0)$-lift.

    \begin{proposition}[Lift consequences]\label{prop: lift consequences}
        Fix $n \geq 3$ and consider $X \in \mathbf H^{(n)}$ with an $(s,\nu)$-lift. Then 
        \begin{align}\label{eq: dimension consequence}
            d_{\Kak}(X^{(n)}) \geq d_{\lin}(n+s-\nu) - s.
        \end{align}
        In particular if the linear Kakeya set conjecture holds in dimension $n+s-\nu$, then $d_{\Kak}(X^{(n)}) \geq n-\nu$. If moreover $\nu = 0$, then $d_{\Kak}(X^{(n)}) = n$. 

        Suppose $\nu = 0$. There is a compact ball $\widetilde Y \subset \R^{N-1}$ such that for every $p \in [1,\infty)$,
        \begin{align}\label{eq: max fnc consequence}
            \|\mathcal K_{\delta}^{X^{(n)}}\|_{L^p(\R^n) \to L^p(Y)} \lesssim  \|\mathcal K_\delta^{\lin, n+s} \|_{L^p(\R^{n+s}) \to L^p(\widetilde Y)}.
        \end{align}
    \end{proposition} 

    \begin{remark}
        We will see that there are many $X$ with $(s,0)$-lifts that fail Bourgain's condition, and hence fail the critical $L^n \to L^n$ maximal function estimate. The estimate \eqref{eq: max fnc consequence} is still consistent with the linear Kakeya maximal conjecture in $\R^{n+s}$ because the critical exponent there is $p = n+s$, not $p = n$. 
    \end{remark}

    \begin{remark}
        One can prove a version of \eqref{eq: max fnc consequence} with $\nu > 0$, but we leave those details to the interested reader. The estimate \eqref{eq: max fnc consequence} is not used in the main arguments---its purpose is to demonstrate the quantitative strength of our notion of a lift. 
    \end{remark}

    The following Proposition records the lifts of the families of curves that lead to Theorems \ref{thm: density}, \ref{thm: quadratic}, and \ref{thm: tan example main}.

    \begin{proposition}[Examples of lifts of families of curves] \label{prop: lifting examples}\hspace{0pt} 
        \begin{enumerate}[label=(\roman*)]
            \item (Quadratic) Fix $n \geq 3$. The defining function $X_B$ from \eqref{eq: quadratic def} has a $(\rank B, \rank B - \rank (B^2))$-lift.
            \item (Polynomial) Fix $n\geq 3$, $D \geq 1$, compact balls $W,Y \subset \R^{n-1}$, and a compact interval $I \subset \R$. Consider the defining function $X_D : W \times I \times Y \to \R^{n-1}$ given by 
            \begin{align}
                X_D(w,t;y) := A_0(w,y) + tA_1(w,y) + \cdots + t^{D-1} A_{D-1}(w,y) + t^D B(y),
            \end{align}
            where the $A_i$ are smooth functions, $\det DB \neq 0$, and $X_D$ is H\"ormander-type. Then $X_D$ has a $((D-1)(n-1), 0)$-lift.
            \item ($\tan$-example) The defining function $X_{\tan}$ from \eqref{eq: tan example def} has an $(n,0)$-lift. 
        \end{enumerate}
    \end{proposition}

    We prove Proposition \ref{prop: lift consequences} in the next section. Theorem \ref{thm: tan example main} follows immediately from Proposition \ref{prop: lifting examples}(iii) and Proposition \ref{prop: lift consequences}. Theorems \ref{thm: density} and \ref{thm: quadratic} require extra work, and they occupy their own sections.

    \section{Proof of Proposition \ref{prop: lift consequences}} \label{sec: prop: lift consequences}

    We first prove the Hausdorff dimension statement \eqref{eq: dimension consequence}.

    \begin{proof}[Proof of \eqref{eq: dimension consequence}]

    Let $K$ be an $X$-Kakeya set. For each $y \in Y$, choose $w(y) \in W$ such that $\gamma_{y,w(y)} \subset  K$.
    Fix a compact set $\Omega_0 \subset \Omega$ containing every lifted segment $L_{y,w,\rho}$, and define 
    \begin{align}\label{eq: tilde K defn}
        \widetilde K = \pi^{-1}(K) \cap \Omega_0.
    \end{align}
    By \eqref{eq: projection property}, 
    \begin{align}
        L_{y,w(y), \rho} \subset \widetilde K, \quad y \in Y, \rho \in R.
    \end{align}
    Since $\pi : \Omega \to \R^n$ is a submersion with $s$-dimensional fibers, it is locally bi-Lipschitz equivalent to the coordinate projection $\R^n \times \R^s \to \R^n$. Using the standard identity $\dim_\H(E \times [0,1]^s) = \dim_{\H} E + s$ and \eqref{eq: tilde K defn}, we obtain 
    \begin{align}\label{eq: step 1 of dim statement}
        \dim_\H \widetilde K \leq \dim_\H K + s.
    \end{align}
    
    At this stage the set of directions in $\widetilde K$ has dimension $N-1-\nu$. We project $\widetilde K$ onto $N - \nu$ dimensions so that we get a full Kakeya set in $\R^{N - \nu}$. Note that this step is only necessary when $\nu \geq 1$. 
    Define $r = N-1 -\nu$.
    Choose an interior point $(y_0,\rho_0) \in Y \times R$. Since $\rank D\Theta(y_0,\rho_0) = r$, there is a coordinate projection $P : \R^{N-1} \to \R^r$ such that 
    \begin{align}
        \rank D(P \circ \Theta )(y_0,\rho_0) = r.
    \end{align}
    By the submersion theorem, the image of $P \circ \Theta : Y \times R \to \R^{r}$ contains a nonempty open set $V \subset \R^r$. Define the linear projection $\Pi : \R^{N-1} \times \R \to \R^r \times \R$, $\Pi(x,t) = (P(x),t)$. For every $(y,\rho)$ in a small neighborhood of $(y_0,\rho_0)$, 
    \begin{align}
        \Pi(L_{y,w(y),\rho}) = \{ (P(a(y,w(y),\rho)) + t P(\Theta(y,\rho)), t) : t \in I\} \subset \Pi(\widetilde K).
    \end{align}
    Thus $\Pi(\widetilde K)$ contains a line segment in an open set of directions on $S^r$, so up to taking finitely many rotated copies of $\Pi(\widetilde K)$, we obtain a Kakeya set in $\R^{r+1}$. Consequently
    \begin{align}
        d_{\lin}(r + 1) \leq \dim_\H \Pi(\widetilde K) \leq \dim_\H \widetilde K.
    \end{align}
    Using $r + 1 = n + s -\nu$ and \eqref{eq: step 1 of dim statement}, we obtain 
    \begin{align}
        d_{\lin}(n+s-\nu) \leq \dim_\H K + s.
    \end{align}
    Taking an infimum over all $X$-Kakeya sets $K$ proves \eqref{eq: dimension consequence}.
    \end{proof}

    Now we prove the maximal function statement \eqref{eq: max fnc consequence}, which is the final part of Proposition \ref{prop: lift consequences} left to be shown.

    \begin{proof}[Proof of \eqref{eq: max fnc consequence}]
    Assume $\nu = 0$ and set $N = n+s$.
    Let $\chi$ be the indicator of $\Omega_0$, a compact set containing all the segments $L_{y,w,\rho}$ for $(y,w,\rho) \in Y \times W \times R$. For $f \in L^p(\R^n)$, define $\tilde f = (f \circ \pi) \chi \in L^p(\R^N)$. Let $J_\pi$ be the $n$-dimensional Jacobian of $\pi$. Since $\pi$ is a submersion, $J_\pi \sim 1$. The coarea formula then gives 
    \begin{align}
        \|\tilde f\|_{L^p(\R^N)}^p &\lesssim \int_{\R^n} |f(x)|^p \mathcal H^s(\pi^{-1}(x) \cap \Omega_0) dx \\
        &\lesssim \|f\|^p_{L^p(\R^n)}.\label{eq: lp comparison fnc}
    \end{align}
    Here $\mathcal H^s$ is the $s$-dimensional Hausdorff measure.
    Now we compare the curved tubes with the lifted linear tubes. By compactness and the local normal form for $\pi$ as a projection, there exists a constant $A \geq 1$ such that for sufficiently small $\delta$ we have 
    \begin{align}\label{eq: fiber estimate}
        \mathcal H^s(\pi^{-1}(x) \cap N_{A\delta}(L_{y,w,\rho}) \cap \Omega_0) \gtrsim \delta^s,
    \end{align}
    whenever $x \in N_\delta(\gamma_{y,w})$.
    Indeed in local coordinates where $\pi$ is a projection, the set on the left contains an $s$-dimensional ball of radius $\sim \delta$. By the coarea formula and \eqref{eq: fiber estimate}, 
    \begin{align}
        \int_{N_{A\delta}(L_{y,w,\rho})} |\tilde f| \gtrsim \delta^s \int_{N_\delta(\gamma_{y,w})} |f|.
    \end{align}
    We also have $|N_{\delta}(\gamma_{y,w})| \sim \delta^{n-1}$ and $|N_{A\delta}(L_{y,w,\rho})| \sim \delta^{N-1}$. Thus as $N-1 = n-1+s$, we obtain 
    \begin{align}
        \frac{1}{|N_\delta(\gamma_{y,w})|} \int_{N_\delta(\gamma_{y,w})} |f| \lesssim \frac{1}{|N_{A\delta}(L_{y,w,\rho})|} \int_{N_{A\delta}(L_{y,w,\rho})} |\tilde f|.
    \end{align}
    After covering the tube $N_{A\delta}(L_{y,w,\rho})$ with $O_A(1)$ parallel $\delta$-tubes and taking a supremum over $w \in W$, we obtain
    \begin{align}\label{eq: ptwise max fnc est}
        \mathcal K_\delta^X f(y) \lesssim \mathcal K_{\delta}^{\lin, N} \tilde f(\Theta(y,\rho)).
    \end{align}
    Let $\widetilde Y$ be a compact ball containing $\Theta(Y \times R)$.
    Raising both sides of \eqref{eq: ptwise max fnc est} to $p$, integrating over $Y \times R$, and using that $\Theta : Y \times R \to \R^{N-1}$ has full rank, we obtain 
    \begin{align}
        \int_Y |\mathcal K_\delta^X f(y)|^p dy &\lesssim \int_{Y \times R} |\mathcal K_\delta^{\lin, N} \tilde f(\Theta(y,\rho))|^p dy d\rho  \\
        &\lesssim \int_{\widetilde Y} |\mathcal K_{\delta}^{\lin, N} \tilde f(\theta)|^p d\theta.
    \end{align}
    Thus
    \begin{align}
        \| \mathcal K_\delta^{X} f\|_{L^p(Y)} \lesssim \|\mathcal K_{\delta}^{\lin, N} \tilde f \|_{L^p(\widetilde Y)} \leq \|\mathcal K_\delta^{\lin,N}\|_{L^p(\R^N)\to L^p(\widetilde Y)} \| \tilde f\|_{L^p(\R^N)}.
    \end{align}
    Together with \eqref{eq: lp comparison fnc}, this proves \eqref{eq: max fnc consequence}. 
    \end{proof}

    \section{Proof of Proposition \ref{prop: lifting examples}}\label{sec: prop: lifting examples}

    We begin with Proposition \ref{prop: lifting examples}(i), which is the most concrete. 
    
    \begin{proof}[Proof of Proposition \ref{prop: lifting examples}(i)]
        Let $B$ be an $(n-1)\times (n-1)$ matrix, let $(s,\nu) = (\rank B, \rank B - \rank(B^2))$, and define $q = \rank B$. Consider the rank factorization 
        \begin{align}
            B = UV, \quad U : \R^q \to \R^{n-1} \text{ injective}, \quad V : \R^{n-1} \to \R^q\text{ surjective.}
         \end{align}
         Consider the coordinates $(x,z,t) \in \R^{n-1} \times \R^q \times \R = \R^{N}$ and define 
         \begin{align}
             \pi(x,z,t) = (x + tUz,t) \in \R^{n-1} \times \R.
         \end{align}
         This is a submersion because $D\pi$ has rank $n$. Let $R \subset \R^q$ be a compact ball. For $\rho \in R$, set 
         \begin{align}
             a(y,w,\rho) = (w,\rho), \quad \Theta(y,\rho) = (y- U\rho,Vy).
         \end{align}
         We may compute 
         \begin{align}
             \pi((a,0) + t(\Theta,1)) &= (w + t(y-U\rho) + tU(\rho + tVy),t) \\
             &= (w + (tI +t^2B)y,t),
         \end{align}
         so \eqref{eq: projection property} holds.

         Now we show that $\rank D\Theta = N - 1 - \nu = (n-1) + \rank (B^2)$. The kernel of $D \Theta$ consists of pairs $(h,k)$ satisfying
         \begin{align}
            h - Uk = 0, \quad Vh=0,
         \end{align}
         so $\ker D\Theta$ is naturally identified with $\ker(VU)$. Thus $\dim\ker (VU) = \dim \ker(D\Theta)$.
         Since $U$ is injective and $V$ is surjective, 
         \begin{align}
             \rank(VU) = \rank(UVUV) = \rank(B^2).
         \end{align}
         Then by applying rank-nullity to $VU$ and to $D\Theta$, we obtain 
         \begin{align}
             \rank D\Theta &= (n-1) + \rank(B^2) \\
             &= N-1-\nu.
         \end{align}
    \end{proof}

    Now we move on to Proposition \ref{prop: lifting examples}(ii).
    \begin{proof}[Proof of Proposition \ref{prop: lifting examples}(ii)]
        Consider 
        \begin{align}
            X_D(w,t;y) = \sum_{j=0}^{D-1} t^j A_j(w,y) + t^D B(y), \quad \det DB(y) \neq 0.
        \end{align}
        Set $s = (D-1)(n-1)$, so $N-1=D(n-1)$. Use the coordinates $(z_0,\ldots, z_{D-1}) \in (\R^{n-1})^D$. Define the submersion 
        \begin{align}
            \pi(z_0,\ldots, z_{D-1},t) := (\sum_{j=0}^{D-1} t^j z_j, t).
        \end{align}
        Write $\rho = (\rho_1,\ldots, \rho_{D-1})$ and define 
        \begin{align}
            a(y,w,\rho) &:= (A_0(w,y), A_1 (w,y)-\rho_1, \ldots, A_{D-1}(w,y)-\rho_{D-1}) \\
            \Theta(y,\rho) &:= (\rho_1,\ldots, \rho_{D-1}, B(y)).
        \end{align}
        We have the telescoping sum 
        \begin{align}
            \sum_{j=0}^{D-1} t^j (a_j + t\Theta_j) 
            &= A_0 + \sum_{j=1}^{D-1} t^j (A_j - \rho_j) + \sum_{j=1}^{D-1} t^j \rho_j + t^D B(y) \\
            &= X_D(w,t;y).
        \end{align}
        Thus we obtain \eqref{eq: projection property}:
        \begin{align}
            \pi((a,0)+t(\Theta,1)) = \gamma_{y,w}(t).
        \end{align}
        We have that 
        \begin{align}
            D\Theta(y,\rho) = 
            \begin{pmatrix}
                0 & I_{(D-1)(n-1)} \\
                DB(y) & 0
            \end{pmatrix}
        \end{align}
        is invertible because $DB(y)$ is invertible. Thus $X_D$ has a $((D-1)(n-1), 0)$-lift.
    \end{proof}

    Finally we handle Proposition \ref{prop: lifting examples}(iii). 
    \begin{proof}[Proof of Proposition \ref{prop: lifting examples}(iii)]
        Set $k = n-2$ and write $y = (y', y_{n-1}) \in \R^k \times \R$, $w = (w', w_{n-1})$. Recall that 
        \begin{align}
            X_{\tan}(w,t;y) = (w' - t^2y', \tan^{-1}(w_{n-1}/t)-ty_{n-1}).
        \end{align}
        Consider the coordinates $(b,c,u,v,q,t) \in \R^k \times \R^k \times \R^4 = \R^{2n}$,
        and define
        \begin{align}
            \pi(b, c,u,v,q,t) = (b+tc, \tan^{-1} (u/v) + q, t).
        \end{align}
        This is a submersion, as is immediate from the derivatives in the $b,q$ and $t$ variables.
        Let the ruling parameter be $(r,\alpha, \lambda) \in \R^k \times \R \times \R$ and choose the compact ball $R$ centered at $(0,0,1)$ sufficiently small; radius $1/100$ will do. Define
        \begin{align}
            a(y,w,r,\alpha,\lambda) &:= (w',r,\lambda w_{n-1} \cos\alpha, -\lambda w_{n-1} \sin \alpha, -\alpha), \\
            \Theta(y,r,\alpha,\lambda) &:= (-r,-y', \lambda \sin \alpha, \lambda \cos \alpha, - y_{n-1}).
        \end{align}
        Then
        \begin{align}
            &\pi(L_{y,w,r,\alpha,\lambda}(t)) \\&= \pi(w'-tr,r-ty',\lambda(w_{n-1} \cos\alpha+t\sin\alpha),\lambda(-w_{n-1}\sin \alpha + t\cos \alpha),-\alpha-ty_{n-1},t) \\
            &=(w'-t^2y',\tan^{-1}(\frac{\lambda(w_{n-1}\cos \alpha + t \sin \alpha)}{\lambda(- w_{n-1}\sin \alpha +t\cos \alpha )})-\alpha - ty_{n-1},t). \label{eq: tan computation}
        \end{align}
        Let us handle the $\tan^{-1}$-term. Set $\beta = \tan^{-1}(w_{n-1}/t)$, so $\tan \beta = w_{n-1}/t$. The tangent addition formula gives 
        \begin{align}
            \frac{\lambda(w_{n-1}\cos \alpha + t \sin \alpha)}{\lambda(- w_{n-1}\sin \alpha +t\cos \alpha )} = \frac{\tan \beta + \tan \alpha}{1 - \tan \beta \tan \alpha} 
            = \tan(\beta + \alpha),
        \end{align}
        and therefore 
        \begin{align}
            \tan^{-1}(\frac{\lambda(w_{n-1}\cos \alpha + t \sin \alpha)}{\lambda(- w_{n-1}\sin \alpha +t\cos \alpha )})
            = \tan^{-1}(w_{n-1}/t) + \alpha. \label{eq: tan computation 2}
        \end{align}
        We used that $\beta + \alpha \in (-\pi/2, \pi/2)$, which is guaranteed by our range of parameters. Plugging \eqref{eq: tan computation 2} back into \eqref{eq: tan computation} gives $\pi(L_{y,w,r,\alpha,\lambda}(t)) = \gamma_{y,w}(t)$. 

        It remains to check the direction rank. With the variables ordered as $(y',y_{n-1}, r,\alpha,\lambda)$, the $y', y_{n-1}$, and $r$ blocks of $D\Theta$ are identity matrices up to a sign, while the $(\alpha,\lambda)$ block is 
        \begin{align}
            \begin{pmatrix}
            \lambda \cos \alpha & \sin \alpha \\
            -\lambda \sin \alpha & \cos \alpha
            \end{pmatrix},
        \end{align}
        whose determinant is $\lambda \neq 0$. Thus $D\Theta$ is invertible, and $X_{\tan}$ has an $(n,0)$-lift.
    \end{proof}

     \section{Proof of Theorem \ref{thm: density}}\label{sec: thm: density}

    There are two main components. First we identify the dense set $\mathbf P \subset \mathbf H$ so that conditional on linear Kakeya in all dimensions, $d_{\Kak}(X^{(n)}) = n$ for each $X \in \mathbf P$. Second, we prove that $\mathbf B \subset \mathbf H$ is closed nowhere dense. These are Lemmas \ref{lem: thm: density part 1} and \ref{lem: thm: density part 2} respectively below.

    \subsection{The dense set $\mathbf P \subset \mathbf H$}

    \begin{lemma}\label{lem: thm: density part 1}
        Let $n \geq 3$. There is a dense set $\mathbf P^{(n)} \subset \mathbf H^{(n)}$ such that, conditional on the linear Kakeya set conjecture in all dimensions, $d_{\Kak}(X) = n$ for each $X \in \mathbf P^{(n)}$. 
    \end{lemma}

    \begin{proof}

    For $D \geq 1$, let $\mathbf P_D^{(n)}$ consist of the H\"ormander-type maps 
    \begin{align}
                X_D(w,t;y) := A_0(w,y) + tA_1(w,y) + \cdots + t^{D-1} A_{D-1}(w,y) + t^D B(y),
    \end{align}
    where the $A_i$ are smooth functions and $\det DB(y) \neq 0$. Define 
    \begin{align}
        \mathbf P^{(n)} := \bigcup_{D \geq 1} \mathbf P^{(n)}_D.
    \end{align}

    By Proposition \ref{prop: lifting examples}(ii) and Proposition \ref{prop: lift consequences},
    every $X \in \mathbf P_D^{(n)}$ satisfies 
    \begin{align}
        d_{\Kak}(X) \geq d_{\lin}(n + (D-1)(n-1)) - (D-1)(n-1)
    \end{align}
    Under the linear Kakeya set conjecture, the right-hand side equals $n$, and hence $d_\Kak(X) = n$. 
    
    It remains to prove density. It suffices to show that for each $X \in \mathbf H$, integer $r \geq 1$, and $\eps > 0$, there is $P \in \mathbf P$ such that 
    $\|X-P\|_{C^r(Q)} \leq \eps$. 
    
    Since $\mathbf H \subset C^\infty(Q ; \R^{n-1})$ is open and polynomials are dense in $C^\infty(Q ; \R^{n-1})$, we can find a polynomial $S : \R^{2n-1} \to \R^{n-1}$ such that $S \in \mathbf H$ and 
    $\|X-S\|_{C^r(Q)} \leq \eps/2$. We can then write
    \begin{align}
        S(w,t;y) = \sum_{i=0}^{D-1} t^i A_i(w,y)
    \end{align}
    for some $D \geq 1$, where the $A_i$ are polynomials and in particular are smooth functions. Again using that $\mathbf H$ is open, there is $\delta >0$ and $s$ so that for any perturbation $R$ satisfying $\|R\|_{C^s(Q)} \leq \delta$, one has $S + R \in \mathbf H$. We choose $R(w,t;y) = \eta t^{D} y$ with $\eta > 0$ taken small enough that $\|R\|_{C^{\max(s,r)}(Q)} \leq \min(\delta, \eps/2)$. Set $P := S + R \in \mathbf H$. We therefore have $P \in \mathbf P$ since $P$ takes the form
    \begin{align}
        P(w,t;y) = \sum_{i=0}^{D-1} t^i A_i(w,y) + t^D \eta y,
    \end{align}
    where the leading coefficient $B(y) = \eta y$ satisfies $\det DB = \eta^{n-1} \neq 0$. Also $\|X-P\|_{C^r(Q)} \leq \eps$, so we have shown that $\mathbf P \subset \mathbf H$ is dense. 

    \end{proof}
    
    \subsection{Bourgain's condition is closed nowhere dense}

    \begin{lemma}\label{lem: thm: density part 2}
        Let $n \geq 3$. The set $\mathbf B^{(n)}$ is closed nowhere dense in $\mathbf H^{(n)}$.
    \end{lemma}

    \begin{proof}
        For $X \in \mathbf H$, define $S_X :=(\nabla_w X)^{-1} \nabla_y X$.
        Since $\partial_t S_X$ is invertible by H\"ormander's condition, Bourgain's condition is equivalent to 
        \begin{align} \label{eq: Bourgain equiv cond}
            (\partial_t S_X)^{-1} \partial_t^2 S_X = \lambda I_{n-1},
        \end{align}
        for a smooth scalar function $\lambda$.
        For an $m \times m$ matrix $A$, define the trace-free part by
        $\mathrm{tf}(A) :=A - \frac{\mathrm{tr}(A)}{m} I_m$. Then \eqref{eq: Bourgain equiv cond} is equivalent to $\mathrm{tf}((\partial_t S_X)^{-1} \partial_t^2 S_X)=0$. Define the operator
        \begin{align}
            \mathcal B : \mathbf H \to C^\infty(Q;\mathrm{Mat}_{n-1}(\R)), \quad \mathcal B(X):=\mathrm{tf}((\partial_t S_X)^{-1} \partial_t^2 S_X),
        \end{align}
        The map $\mathcal B$ is continuous on $\mathbf H$ because $A \mapsto \mathrm{tf}(A)$, differentiation, matrix inversion, and matrix multiplication are continuous in $C^\infty$. We have 
        \begin{align}
            \mathbf B = \{X \in \mathbf H : \mathcal B(X) = 0\},
        \end{align}
        so $\mathbf B$ is closed in $\mathbf H$.

        We will now show that $\mathbf B$ has empty interior, and hence is nowhere dense since $\mathbf B$ is also closed. To this end fix $X \in \mathbf B$ and $p_0 = (w_0,t_0;y_0)$ in the interior of $Q$. Define $A_0 = \nabla_w X(p_0), T_0=\partial_t S_X(p_0)$, and choose a non-zero trace-free $(n-1)\times (n-1)$ matrix $M$. For $\eps > 0$, define the perturbation 
        \begin{align}
            X_{\eps}(w,t;y) := X(w,t;y) + \frac{\eps}{2} (t-t_0)^2 A_0 T_0M (y-y_0).
        \end{align}
        For $\eps$ small enough, $X_\eps \in \mathbf H$. We may compute
        \begin{align}
            \partial_t S_{X_\eps}(p_0) = T_0, \quad \partial_t^2 S_{X_\eps}(p_0) = \partial_t^2 S_X(p_0) + \eps T_0 M.
        \end{align}
        Since $X \in \mathbf B$, $\partial_t^2 S_X(p_0) = \lambda_0 T_0$ for some scalar $\lambda_0$. Thus 
        \begin{align}
            \mathcal B(X_\eps)(p_0) &= \mathrm{tf}(\lambda_0 I_{n-1} + \eps M) \\
            &= \eps M \neq 0,
        \end{align}
        so $X_\eps \notin \mathbf B$.
        Since we can take $\eps > 0$ arbitrarily small, there is no open neighborhood of $X$ in $\mathbf H$ which is contained in $\mathbf B$. As $X \in \mathbf B$ was arbitrary, $\mathbf B$ has empty interior in $\mathbf H$, and is therefore nowhere dense.
    \end{proof}

    \section{Proof of Theorem \ref{thm: quadratic}}\label{sec: thm: quadratic}

    Fix an $(n-1) \times (n-1)$ matrix $B$. By Proposition \ref{prop: lifting examples}(i), $X_B^{(n)}$ has a $(\rank B, \rank B - \rank (B^2))$-lift. By Proposition \ref{prop: lift consequences} and the assumed linear Kakeya set conjecture in dimension $n + \rank(B^2)$, we obtain 
    \begin{align}
        d_{\Kak}(X_B) \geq  n - \rank B + \rank(B^2).
    \end{align}

    Now we prove the upper bound $d_{\Kak}(X_B) \leq n - \rank B + \rank (B^2)$ by constructing an $X_B$-Kakeya set contained in a smooth variety of dimension $(n-\rank B + \rank(B^2))$. 
    Set
    \begin{align}
        d = \rank B - \rank (B^2).
    \end{align}
    On a nilpotent Jordan block $J_\ell(0)$, 
    \begin{align}
        \rank J_\ell(0) - \rank (J_\ell(0)^2) = \begin{cases}
            1, & \ell \geq 2, \\
            0, & \ell = 1,
        \end{cases}
    \end{align}
    while every block associated with a nonzero eigenvalue contributes zero. This $d$ is the number of nilpotent Jordan blocks of size at least two. 

    Choose $S \in \mathrm{GL}_{n-1}(\R)$ that puts $B$ into real Jordan form, 
    \begin{align}
        J :=S^{-1}BS = ( \bigoplus_{j=1}^d J_{\ell_j}(0)) \oplus A, \quad \ell_j \geq 2,
    \end{align}
    where $A$ contains all remaining blocks. Let $E_{ij}^{(m)}$ be the $m \times m$ matrix with a $1$ in the $(i,j)$-entry and $0$ everywhere else. Let 
    \begin{align}
        C_J = ( \bigoplus_{j=1}^d E_{\ell_j, \ell_j-1}^{(\ell_j)}) \oplus 0, \quad C=SC_J S^{-1}.
    \end{align}
    Choose $W_B$ large enough that $C([-1,1]^{n-1}) \subset W_B$, and set 
    \begin{align}
        K = \{((C+tI + t^2 B)y,t) : y \in [-1,1]^{n-1}, t \in I_B\}.
    \end{align}
    This is an $X_B$-Kakeya set.

    Consider the coordinates $x' = S^{-1}x$ and $y' = S^{-1}y$. Write $x_{k}'^{(j)}, y_{k}'^{(j)}$ for the coordinates of $x',y'$ corresponding to the $k$-th coordinate of the $j$-th block. Points of $K$ satisfy 
    \begin{align}
        x' = (C_J + tI + t^2 J)y'.
    \end{align}

    For the $j$-th selected Jordan block, 
    \begin{align}
            E_{\ell_j, \ell_{j}-1}^{(\ell_j)} + tI_{\ell_j} + t^2 J_{\ell_j}(0) &= 
            \begin{pmatrix}
                t & t^2 &        &        &        \\
                & t   & \ddots &        &        \\
                &     & \ddots & \ddots &        \\
                &     &        & t      & t^2    \\
                &     &        & 1      & t
            \end{pmatrix}.
        \end{align}
    Hence the final two coordinates in this block obey
    \begin{align}
            x_{\ell_j - 1}'^{(j)} &= ty_{\ell_{j}-1}'^{(j)} + t^2 y_{\ell_j}'^{(j)} \\
            x_{\ell_j}'^{(j)} &= y_{\ell_{j-1}}'^{(j)} + t y_{\ell_j}'^{(j)}.
        \end{align}
    Therefore 
    \begin{align}
        x_{\ell_j-1}'^{(j)} - t x_{\ell_j}'^{(j)} = 0, \quad 1 \leq j \leq d.
    \end{align}
    Define the polynomials 
    \begin{align}
        Q_j(x,t) = (S^{-1} x)_{\ell_{j-1}}^{(j)} - t(S^{-1}x)_{\ell_j}^{(j)}.
    \end{align}
    Then 
    \begin{align}
        K \subset V :=\bigcap_{j=1}^d \{Q_j = 0\}.
    \end{align}
    In the coordinates $x'$, $Q_j$ has derivative 1 with respect to $x_{\ell_j-1}'^{(j)}$, and every other $Q_i$ has derivative $0$ in that coordinate. Thus 
    $\nabla Q_1,\ldots, \nabla Q_d$ are linearly independent everywhere, so $V$ is a smooth variety of codimension $d$ and dimension 
    \begin{align}
        n-d = n - \rank B + \rank(B^2).
    \end{align}
    Consequently,
    \begin{align}
        d_{\Kak}(X_B) \leq n - \rank B + \rank(B^2).
    \end{align}
    Combining this with the lower bound proves the theorem.

\bibliographystyle{alpha}
\bibliography{reference}
	
\end{document}